\documentclass[a4paper]{article}
\usepackage{a4wide}
\usepackage{amsmath, amssymb, amsthm}
\usepackage{thmtools}
\usepackage{microtype,hyperref}
\usepackage{tikz}
\usetikzlibrary{arrows}

\newcommand{\blok}[1]{\leavevmode\unskip\penalty9999 \hbox{}\nobreak\hfill\quad\hbox{$#1$}}
\newcommand{\klaar}{\blok{\blacksquare}}

\declaretheorem[style=plain,numberwithin=section]{theorem}
\declaretheorem[style=plain,sibling=theorem]{lemma}
\declaretheorem[style=plain,sibling=theorem]{corollary}

\declaretheorem[style=definition,prefoothook=\klaar,sibling=theorem]{example}
\declaretheorem[style=definition,sibling=theorem]{remark}

\newcommand{\Z}{\ensuremath{\mathbb{Z}}}

\newcommand{\R}{\ensuremath{\mathbb{R}}}
\newcommand{\transp}{\mathsf{T}}
\newcommand{\clattice}{\mathcal{C}}

\newcommand{\one}{\ensuremath{\mathbf{1}}}

\DeclareMathOperator{\tw}{tw}
\DeclareMathOperator{\Div}{Div}
\DeclareMathOperator{\Prin}{Prin}
\DeclareMathOperator{\Col}{Col}

\DeclareMathOperator{\todiv}{div}

\DeclareMathOperator{\dgon}{gon}
\DeclareMathOperator{\sgon}{sgon}

\DeclareMathOperator{\supp}{supp}
\DeclareMathOperator{\rank}{rank}

\newcommand*{\B}{\mathcal{B}}
\newcommand*{\order}[1]{\left|\left|#1\right|\right|}

\tikzset{punt/.style={fill,circle,inner sep=1.3pt},
         div/.style={draw,circle,minimum width=6mm},
         >=stealth'}

\author{Josse van Dobben de Bruyn\footnote{Leiden University} \and Dion Gijswijt\footnote{Delft University of Technology}}
\title{Treewidth is a lower bound on graph gonality}

\begin{document}
\maketitle

\begin{abstract}
We prove that the (divisorial) gonality of a finite connected graph is lower bounded by its treewidth. We show that equality holds for grid graphs and complete multipartite graphs. 

We prove that the treewidth lower bound also holds for \emph{metric graphs} by constructing for any positive rank divisor on a metric graph $\Gamma$ a positive rank divisor of the same degree on a subdivision of the underlying graph.

Finally, we show that the treewidth lower bound also holds for a related notion of gonality defined by Caporaso and for stable gonality as introduced by Cornelissen et al.
\end{abstract}

\section{Introduction and notation}
Let $G=(V,E)$ be a graph and consider the following chip firing game on $G$. At any stage, we have a chip configuration consisting of a nonnegative number of chips at each vertex of $G$. We may go from one chip configuration to another by a sequence of moves. A move consists of choosing a subset $U\subseteq V$ and moving one chip from $u$ to $v$ for every edge $uv$ with $u\in U$ and $v\in V\setminus U$. For the move to be possible, every vertex $u\in U$ must have at least as many chips as it has edges to vertices in $V\setminus U$. Observe that a move corresponding to a subset $U$ can be reversed by subsequently making the move corresponding to the complementary set $V\setminus U$.

A chip configuration is winning if for every vertex $v$ there is a sequence of moves that results in a configuration with at least one chip on $v$. The gonality $\dgon(G)$ of $G$ is the smallest number of chips in a winning chip configuration.

The main result of this paper is to show that the gonality of a connected graph is lower bounded by its treewidth. This result was conjectured in \cite{vanDobbendeBruyn2012}. 

The remainder of Section 1 is devoted to preliminaries, including basic notation and terminology related to graphs, divisors and treewidth. In Section 2, we state and prove the main theorem. In Section 3, we consider some families of graphs for which treewidth equals gonality. These include: trees, grids and complete multipartite graphs. In Section 4, we briefly review divisor theory for \emph{metric} graphs. We show that the gonality of a metric graph is lower bounded by the gonality of a subdivision of the underlying graph. Hence, the tweewidth is also a lower bound for metric graphs. In Section 5, we discuss some related notions of gonality defined in terms of harmonic morphisms, and show that there the treewidth is also a lower bound.
     
\subsection{Graphs}
The graphs in this paper will be finite and undirected (unless stated otherwise). We allow our graphs to have multiple (parallel) edges, but no loops. We will almost exclusively consider \emph{connected} graphs. For a graph $G$, we denote by $V(G)$ and $E(G)$ the set of vertices and edges of $G$, respectively. By an edge $uv$, we mean an edge with ends $u$ and $v$. For (not necessarily disjoint) subsets $U,W\subseteq V(G)$, we denote by $E(U,W)$ the set of edges with an end in $U$ and an end in $W$. For vertices $u$ and $v$, we use the abbreviations $E(u,v):=E(\{u\},\{v\})$ and $E(u):=E(\{u\},V\setminus \{u\})$. The degree of a vertex $v$ equals the number of edges with $v$ as an endpoint and is denoted by $d_G(v):=|E(v)|$. For a subset $U\subseteq V(G)$, we denote by $G[U]:=(U,E(U,U))$ the subgraph of $G$ \emph{induced by $U$}. That is, $G[U]$ is the graph with vertex set $U$ and edge set consisting of the edges of $G$ with both ends in $U$.

Let $G=(V,E)$ be a \emph{connected} graph. We can make $G$ into an \emph{oriented graph} by, for every edge $e$, assigning one end to be the \emph{head} of $e$ and the other end to be the \emph{tail} of $e$. We view the edge $e$ as oriented from its tail to its head. For a cycle $C$ in $G$, we then denote by $\chi_C\in \R^{E}$ the signed incidence vector defined by
\begin{equation}
\chi_C(e)=\begin{cases}
1& \text{if $e$ is traversed in forward direction by $C$,}\\
-1&\text{if $e$ is traversed in backward direction by $C$,}\\
0&\text{otherwise.}
\end{cases}
\end{equation}    
Similarly, we write $\chi_P$ for the signed incidence vector of a path $P$.

The \emph{incidence matrix} $M=M(G)\in \R^{V\times E}$ of $G$ is defined by, for every $v\in V$ and $e\in E$, setting 
\begin{equation}
M_{v,e}:=
\begin{cases}
1&\text{if $v$ is the head of $e$,}\\
-1&\text{if $v$ is the tail of $e$,}\\
0&\text{otherwise.}
\end{cases}
\end{equation}
\noindent The matrix $Q=Q(G):=MM^\transp$ is the \emph{Laplacian} of $G$ and it is independent of the chosen orientation. Indeed, for any two vertices $u$ and $v$, $Q_{vv}$ equals the degree of vertex $v$ and $Q_{uv}$ equals $-|E(u,v)|$. The \emph{cut lattice} of $G$ is the set $\clattice(G):=\Z^{E}\cap \Col(M^\transp)$ of integral vectors in the column space of the transpose of $M$. 

The following two lemma's are well-known, see for example \cite{GodsilRoyle2001}. For the sake of the reader, we will give the short proofs.
 
\begin{lemma}\label{clattice}
Let $f\in \Z^{E}$. Then the following are equivalent:
\begin{itemize}
\item[i)] $f$ is in the cut lattice of $G$,
\item[ii)] $f^\transp\chi_C=0$ for every cycle $C$ in $G$,
\item[iii)] $f=M^\transp x$ for some $x\in \Z^{V}$.
\end{itemize}
\end{lemma} 
\begin{proof}
The implication from iii) to i) is trivial. The implication from i) to ii) follows since $M\chi_C=0$ for every cycle $C$. For the implication from ii) to iii), let $f\in \Z^{E}$ satisfy the condition in ii). Let $T$ be a spanning tree in $G$ with root $r$, and define $x\in \Z^V$ by $x(v):=f^\transp\chi_{P_v}$, where $P_v$ is the path in $T$ from $r$ to $v$. Now for every edge $e=uv$ oriented from $u$ to $v$, we have $x(v)-x(u)=f(e)$. Hence, $f=M^\transp x$.
\end{proof}

\begin{lemma}\label{Laplacian}
The null space of $Q$ is spanned by the all-one vector $\one$.
\end{lemma}
\begin{proof}
Since the row sums of $Q$ equal zero, it is clear that $Q\one=0$. Conversely, let $x$ be in the null space of $Q$ and suppose, for contradiction, that $x$ is not a multiple of $\one$. Since $G$ is connected, we may choose $v\in V$ for which $x(v)$ is maximal and such that $v$ has a neighbour $u$ with $x(u) < x(v)$. From $Qx=0$ it follows that $d_G(v)x(v)=\sum_{w\in V\setminus\{v\}} |E(v,w)|\cdot x(w)$. On the other hand,
\begin{equation*}
\sum_{w\in V\setminus\{v\}}|E(v,w)|\cdot x(w)\ < \sum_{w\in V\setminus\{v\}}|E(v,w)|\cdot x(v)= d_G(v)x(v)
\end{equation*}
by our choice of $v$. This is a contradiction.  
\end{proof}

\subsection{Divisors}
We will largely adopt notation from \cite{BakerNorine2007}. Let $G=(V,E)$ be a connected graph. A vector $D\in \Z^V$ is called a \emph{divisor} on $G$. The set $\Div(G):=\Z^V$ denotes the set of all divisors on $G$. For a divisor $D\in \Div(G)$ we call $\deg(D):=\sum_{v\in V}D(v)$ the \emph{degree} of $D$. A divisor $D$ is said to be \emph{effective} if it is nonnegative. We denote by $\Div_+(G)$ the set of effective divisors on $G$ and by $\Div_+^k(G)$ the set of effective divisors of degree $k$. We denote by $\supp(D):=\{v\in V\mid D(v)\neq 0\}$ the \emph{support} of $D$.

We call two divisors $D$ and $D'$ \emph{equivalent} and write $D\sim D'$ if there is an integer vector $x\in \Z^V$ such that $D-D'=Q(G)x$. Clearly, this is indeed an equivalence relation. Observe that equivalent divisors have equal rank as $Q(G)$ has column sums equal to zero. 

We will often consider the situation where $x$ is the incidence vector of a subset $U$ of $V$, that is $D'=D-Q(G)\one_U$. Observe that is this case 
\begin{equation}
D'(v)=\begin{cases}D(v)-|E(\{v\},V\setminus U)|&\text{if $v\in U$},\\D(v)+|E(\{v\},U)|&\text{if $v\in V\setminus U$.}\end{cases}
\end{equation}
In particular, $D'(v)\leq D(v)$ if $v\in U$ and $D'(v)\geq D(v)$ if $v\in V\setminus U$. In terms of chip firing, we move one chip along each edge in the cut $E(U,V\setminus U)$. The following lemma shows that for equivalent effective divisors $D$ and $D'$, we can obtain $D'$ from $D$ by a sequence of steps of this form and each intermediate divisor being effective.

\begin{lemma}\label{chain}
Let $D_0$ and $D$ be equivalent effective divisors satisfying $D \neq D_0$. There is a chain of sets $\emptyset\subsetneq U_1\subseteq U_2\subseteq \cdots\subseteq U_k\subsetneq V$ such that $D_t:=D-Q(G)(\sum_{i=1}^t \one_{U_i})$ is effective for every $t=1,\ldots,k$ and $D_k=D$. Moreover, this chain is unique.
\end{lemma}
\begin{proof}
Since $D_0\sim D$, there exists an $x\in \Z^V$ such that $D_0-Q(G)x=D$. By Lemma \ref{Laplacian}, $x$ is unique up to integral multiples of $\one$. Hence, there is a unique such $x$ with the additional property that $x\geq 0$ and $\supp(x)\neq V$. Let $k:=\max\{x(v)\mid v\in V\}$ and define $U_i:=\{v\in V\mid x(v)\geq k-i+1\}$ for $i=1,\ldots, k$. It follows that $\sum_{i=1}^k \one_{U_i}=x$.

Now consider any $v\in V$ and any $t\in \{1,\ldots, k\}$. If $v\not\in U_t$, then $D_0(v)\leq D_1(v)\leq \cdots\leq D_t(v)$, hence $D_t(v)\geq 0$. If $v\in U_t$, then $D_t(v)\geq D_{t+1}(v)\geq \cdots \geq D_k(v)$, hence $D_t(v)\geq 0$. It follows that $D_1,\ldots,D_{k-1}$ are effective.

Uniqueness follows directly from the uniqueness of an $x\in \Z^V$ for which $x\geq 0$, $\supp(x)\neq V$ and $D_0-Q(G)x=D$, in combination with the uniqueness of the decomposition $x=\sum_{i=1}^k\one_{U_i}$ as a sum of characteristic vectors of a chain $\emptyset\subsetneq U_1\subseteq U_2\subseteq \cdots\subseteq U_k\subsetneq V$.
\end{proof}

The \emph{rank} of a divisor $D$ is defined as
\begin{equation}
\rank(D):=\max\{k\mid \text{$D-D'$ is equivalent to an effective divisor for every $D'\in\Div_+^k$}\}.
\end{equation}

\noindent Observe that equivalent divisors have equal rank and that $\rank(D)\leq \deg(D)$. Also observe that the restriction of $D'$ to effective divisors in the definition is immaterial. 

Following Baker \cite{Baker2008}, we define the \emph{gonality} of $G$ by
\begin{equation}
\dgon(G):=\min\{k\mid \text{there is a divisor of degree $k$ on $G$ with positive rank}\}.
\end{equation} 


An effective divisor $D$ is called \emph{$v$-reduced} if for any nonempty subset $U\subseteq V\setminus \{v\}$ the divisor $D-Q(G)\one_U$ is not effective. In other words, for every nonempty $U\subseteq V\setminus \{v\}$ there is a $u\in U$ with $D(u)<|E(\{u\}, V\setminus U)|$. 

\begin{lemma}\label{reduction}
Let $v\in V$ and let $D$ be an effective divisor on $G$. Then there is a unique $v$-reduced divisor equivalent to $D$.
\end{lemma}
\begin{proof}
For any divisor $D' \sim D$, there is a unique $x_{D'}\in \{x\in \Z^V\mid x\geq 0,\ x(v)=0\}$ such that $D'=D-Q(G)x_{D'}$ by Lemma \ref{Laplacian}. Let 
\begin{equation}
S:=\{x_{D'}\mid D'\text{ is effective and equivalent to $D$}\}.
\end{equation}
The set $S$ is finite since the number of effective divisors equivalent to $D$ is finite. Choose $x_{D'}\in S$ maximizing $\sum_{u\in V} x_{D'}(u)$. Then $D'$ is $v$-reduced because for any nonempty $U\subseteq V\setminus \{v\}$, the vector $x_{D'}+\one_U$ is not in $S$ by the choice of $x_{D'}$. 

To show uniqueness, let $D$ and $D'$ be two different, but equivalent effective divisors. It suffices to show that $D$ and $D'$ are not both $v$-reduced. By Lemma \ref{chain} there are sets $\emptyset\subsetneq U_1\subseteq\cdots\subseteq U_k\subsetneq V$ such that $D-Q(G)\one_{U_1}$ and $D'+Q(G)\one_{U_k}=D'-Q(G)\one_{V\setminus U_k}$ are effective. If $v\not\in U_1$, then $D$ is not $v$-reduced. If $v\in U_1$, then $v\not\in V\setminus U_k\subseteq V\setminus U_1$ and hence $D'$ is not $v$-reduced.
\end{proof}
\noindent Observe that if $D$ is $v$-reduced and $\rank(D)\geq 1$, then we have $D(v)\geq 1$.

We say that a divisor $D$ \emph{covers} $v\in V$ if there is an effective divisor $D'$ equivalent to $D$ with $v\in \supp(D')$. A nonempty subset $S\subseteq V$ is called a \emph{strong separator} if for each component $C$ of $G[V\setminus S]$ we have that $C$ is a tree and $|E(\{s\},V(C))|\leq 1$ for every $s\in S$. The folowing lemma is similar to a theorem of Luo \cite{Luo2011} on rank determining sets in the context of metric graphs.

\begin{lemma}\label{Luotype}
Let $S$ be a strong separator of $G$ and let $D$ be a divisor covering every $s\in S$. Then $D$ has positive rank. 
\end{lemma}
\begin{proof}
Since any superset of a strong separator is again a strong separator, we may assume that $S=\{s\in V\mid \text{ $s$ is covered by $D$}\}$. We have to show that $S=V$. 

Suppose not. Let $C$ be a component of $G[V\setminus S]$ and let $S':=\{s\in S: |E(\{s\},V(C))|=1\}$. Since $G$ is connected, $S'$ is not empty, so we may take $s\in S'$ and assume that $D$ is $s$-reduced. If $S'\subseteq \supp(D)$, then $D+Q(G)\one_{V(C)}$ is effective and has support on at least one vertex in $V(C)\subseteq V\setminus S$, a contradiction. 

Hence, we may assume that there is a $t\in S'\setminus \supp(D)$. In particular, $D$ is not $t$-reduced. Let $a$ and $b$ be the unique neighbours of $s$ and $t$ in $V(C)$, respectively, and let $P=(s,a,\ldots,b,t)$ be the path from $s$ to $t$ with its interior points in $V(C)$. Since $D$ is $s$-reduced, but not $t$-reduced, there is a set $U\subseteq V$ with $s\in U$, $t\not\in U$ such that $D':=D-Q(G)\one_U$ is effective. The cut $E(U,V\setminus U)$ must intersect some edge $e=uv$ of the path $P$, and we find that $D(u)\geq 1$ and $D'(v)\geq 1$. Since at least one of $u$ and $v$ is in $V(C)\subseteq V\setminus S$, we obtain a contradiction.
\end{proof}

\begin{corollary}\label{rankdetermining}
If $H$ is a subdivision of $G$ and $D$ is a divisor on $H$ that covers all $v\in V(G)$, then $D$ has positive rank.
\end{corollary}

\subsection{Treewidth}
The notion of \emph{treewidth} was first introduced by Halin \cite{Halin1976} and later rediscovered by Robertson and Seymour \cite{RobertsonSeymour1990} as part of their graph minor theory. There are several equivalent definitions of treewidth. The most natural one is perhaps in terms of tree-decompositions of a graph. However, for reasons of brevity and since we will not need tree-decompositions here, we use the following definition in terms of chordal extensions. 

A graph $H$ is called \emph{chordal} if it has no induced cycle of length at least 4. If $G=(V,E)$ is a subgraph of a chordal graph $H=(V,F)$, then $H$ is called a \emph{chordal extension} of $G$. We denote the maximum size of a clique in a graph $H$ by $\omega(H)$. The treewidth $\tw(G)$ of a graph $G$ can now be defined by 
\begin{equation}
\tw(G):=-1+\min\{\omega(H)\mid \text{ $H$ is a chordal extension of $G$}\}.
\end{equation}
\noindent Observe that the treewidth of a (nontrivial) tree equals 1 and the treewidth of a complete graph on $n$ nodes equals $n-1$ as these graphs are chordal and have clique number $2$ and $n$, respectively. 

It is NP-complete to determine for a given graph $G$ and a given integer $k$ whether $\tw(G)\leq k$ (see  \cite{ArnborgEtal1987}). The fact that this problem is in NP follows directly from the definition by using a suitable chordal extension $H$ as a certificate. Indeed, a perfect elimination order for $H$ certifies chordality of $H$ and provides $\omega(H)$.

In order to use treewidth as a lower bound, we will need a way to lower bound treewidth. For this, we will utilize the notion of bramble. Let $G=(V,E)$ be a graph, and let $2^V$ denote the power set of $V$. A set $\B\subseteq 2^V\setminus\{\emptyset\}$ is called a \emph{bramble} if for any $B,B'\in \B$ the induced subgraph $G[B\cup B']$ is connected. In particular, $G[B]$ is connected for every $B\in \B$. For any $B,B'\in \B$, either $B\cap B'\neq \emptyset$, or $B\cap B'=\emptyset$ and there is an edge in $E(B,B')$. In the latter case, we say that $B$ and $B'$ \emph{touch}.
A set $S\subseteq V$ is called a \emph{hitting set} for $\B$ if it has nonempty intersection with every member of $\B$. The \emph{order} of $\B$, denoted $\order{\B}$, is the minimum size of a hitting set for $\B$. That is:
\begin{equation}
\order{\B}:=\min \{|S| : S\subseteq V, S\cap B\neq\emptyset \text{ for all $B\in \B$}\}.
\end{equation}

We will use the following characterization of treewidth due to Seymour and Thomas~\cite{SeymourThomas1993}.

\begin{theorem}[treewidth duality]
Let $k\geq 0$ be an integer. A graph $G$ has treewidth at least $k$ if and only if it has a bramble of order at least $k + 1$.
\end{theorem} 

\begin{remark}
Observe that the treewidth of a graph is equal to the treewidth of the underlying simple graph. It is well-known that treewidth is monotone under taking minors (see for example \cite{DiestelGraphTheory}). That is, removing edges or contracting edges can only decrease treewidth. This also follows easily from the definition. 

In particular, if $H$ is a subdivision of $G$, then $\tw(G)\leq \tw(H)$. It is not hard to see that if $G$ has treewidth at least 2, then in fact $\tw(G)=\tw(H)$ holds. Indeed, it suffices to consider the case that $H$ is obtained from $G$ by subdividing an edge $uv$. Let $G'$ be a chordal extension of $G$ with $\omega(G')=\tw(G)+1$. By adding to $G'$ a new node $w$ and edges $uw$ and $vw$, we obtain a chordal extension $H'$ of $H$. Clearly, $\omega(H')=\max(3,\omega(G'))$. Hence,  
\begin{equation}
\tw(H)\leq \omega(H')-1=\max(2,\omega(G')-1)=\max(2,\tw(G))=\tw(G).
\end{equation}  
If $\tw(G)=1$ and $G$ has two parallel edges, then subdividing such an edge yields a graph of treewidth 2.     
\end{remark}
We refer the interested reader to Chapter 12 in \cite{DiestelGraphTheory} for an excellent exposition of treewidth and its role in the graph minor theory.

\section{Proof of the main theorem}
In this section we prove our main theorem.
\begin{theorem}\label{main}
Let $G=(V,E)$ be a connected graph. Then $\dgon(G)\geq \tw(G)$.
\end{theorem} 

We start by stating and proving two lemmas.
\begin{lemma}\label{oneside}
Let $D,D'$ be effective divisors such that $D'=D-Q(G)\one_U$ for some subset $U\subseteq V$. Let $B\subseteq V$ be such that $G[B]$ is connected. Suppose that $B\cap \supp(D)$ is nonempty, but $B\cap \supp(D')$ is empty. Then $B\subseteq U$.
\end{lemma}
\begin{proof}
Clearly, $B$ cannot be a subset of $V\setminus U$, because otherwise $D'(v)\geq D(v)$ for every $v\in B$. Now suppose that $B\cap U$ and $B\setminus U$ are both nonempty. Since $G[B]$ is connected, there is an edge $uv$ with $u\in B\cap U$ and $v\in B\setminus U$. But then $D'(v)=(D-Q(G)\one_U)(v)\geq D(v)+1\geq 1$ since $u\in U$ is a neighbour of $v\in V\setminus U$. This is a contradiction as well, so we see that $B\subseteq U$ must hold.
\end{proof}

\begin{lemma}\label{brambleseparation}
Let $\B$ be a bramble in $G$ and let $U\subseteq V$. Suppose that there exist $B,B'\in \B$ such that $B\subseteq V\setminus U$ and $B'\subseteq U$. Then $|E(U,V\setminus U)|+1\geq \order{\B}$.
\end{lemma}

\begin{proof}
We will construct a hitting set for $\B$ of size at most $|E(U,V\setminus U)|+1$.
Let $F:=E(U,V\setminus U)$ be the cut determined by $U$ and let $H:=(V,F)$. Let $$X:=\{v\in U\mid d_H(v)\geq 1\}\qquad\text{and}\qquad Y:=\{v\in V\setminus U\mid d_H(v)\geq 1\}$$
be the `shores' of the cut $F$. Let $\B':=\{B'\in \B\mid B'\subseteq U\}$. By assumption, $\B'$ is nonempty. Choose $B'\in \B'$ for which $B'\cap X$ is inclusionwise minimal. Let $B\in \B$ be such that $B\subseteq V\setminus U$. Observe that $B'\cap X$ is nonempty, since $B'$ must touch $B$. 

We now define a hitting set $S$ for $\B$ as follows. Add an arbitrary element $s$ from $B'\cap X$ to $S$. For each edge $xy\in E(X,Y)$ with $x\in X, y\in Y$, we add $x$ to $S$ if $x\not\in B'$, and otherwise we add $y$ to $S$. Hence $|S|\leq 1+|F|$. See Figure \ref{hittingset} for a depiction of the situation.

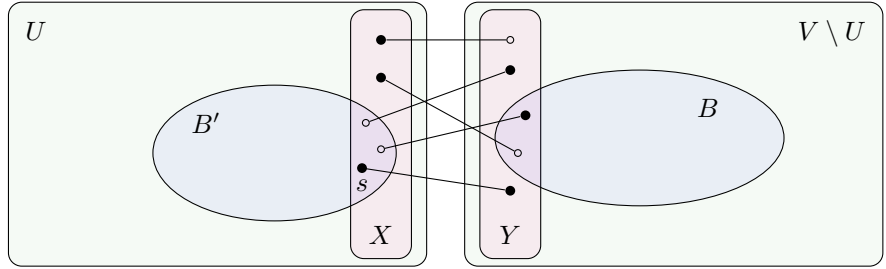
\begin{figure}
\centering
	\begin{tikzpicture}[punt/.style={draw,circle,inner sep=.9pt},
	                    grootpunt/.style={fill,circle,inner sep=1.3pt}]
		\draw[rounded corners=5pt,fill=green!30!gray!6] (0,0) rectangle (5.5,3.5);
		\node[anchor=base west] at (0.1,3) {$U$};
		\draw[rounded corners=5pt,fill=green!30!gray!6] (6,0) rectangle (11.5,3.5);
		\node[anchor=base east] at (11.4,3) {$V\setminus U$};
		\draw[rounded corners=5pt,fill=purple!60!gray!10] (4.5,0.1) rectangle (5.3,3.4);
		\node[anchor=base] at (4.9,0.3) {$X$};
		\draw[rounded corners=5pt,fill=purple!60!gray!10] (6.2,0.1) rectangle (7,3.4);
		\node[anchor=base] at (6.6,0.3) {$Y$};
		\draw[fill=blue,fill opacity=0.05] (3.5,1.5) ellipse (16mm and 9mm);
		\node at (2.6,1.9) {$B'$};
		\draw[fill=blue,fill opacity=0.05] (8.3,1.7) ellipse (19mm and 9mm);
		\node at (9.2,2.1) {$B$};
		\node[grootpunt] (A1) at (4.9,3) {};
		\node[punt] (B1) at (6.6,3) {};
		\draw (A1) -- (B1);
		\node[grootpunt] (B2) at (6.6,2.6) {};
		\node[punt] (A2) at (4.7,1.9) {};
		\draw (B2) -- (A2);
		\node[grootpunt] (B3) at (6.8,2) {};
		\node[punt] (A3) at (4.9,1.55) {};
		\draw (B3) -- (A3);
		\node[grootpunt] (A4) at (4.9,2.5) {};
		\node[punt] (B4) at (6.7,1.5) {};
		\draw (A4) -- (B4);
		\node[grootpunt,label=below:$s$,outer sep=-1pt] (s) at (4.65,1.3) {};
		\node[grootpunt] (B5) at (6.6,1.0) {};
		\draw (s) -- (B5);

	\end{tikzpicture}
\caption{The hitting set $S$ for the bramble $\B$ is formed by the black nodes.\label{hittingset}}
\end{figure}
\medskip

To prove that $S$ covers $\B$, consider any $A\in \B$. First observe that $A$ intersects $X\cup Y$. Otherwise, we would have $A\subseteq U\setminus X$ or $A\subseteq (V\setminus U)\setminus Y$ as $G[A]$ is connected. In the first case $G[A\cup B]$ is not connected and in the second case $G[A\cup B']$ is not connected. In both cases, this contradicts the fact that $\B$ is a bramble.

We consider the following three cases. 
\begin{itemize}
\item Case $A\cap Y=\emptyset$.\quad In this case $A\subseteq U$. By the choice of $B'$, we have either $B'\cap X\subseteq A\cap X$ and hence $s\in A$, or there exists an $x\in (X\cap A)\setminus B'$, which implies that $x\in S$. In both situations $S$ intersects $A$.
\item Case $A\cap X=\emptyset$.\quad In this case $A\subseteq (V\setminus U)$. Since $A$ touches $B'$, there must be an edge $e=xy$ with $x\in B'\cap X$ and $y\in A\cap Y$. By construction of $S$ we have $y\in S$. Hence, $S$ intersects $A$.
\item Case $A\cap X\neq \emptyset$ and $A\cap Y\neq \emptyset$.\quad Since $G[A]$ is connected, there is an edge $e=xy$ with $x\in X$, $y\in Y$ and $x,y\in A$. Since $S$ contains at least one endpoint from each edge in $F$, the set $S$ must intersect $A$.
\end{itemize}
We conclude that $S$ is a hitting set for $\B$ of size at most $|E(U,V\setminus U)|+1$, which proves the lemma.
\end{proof}

\noindent We now prove the main theorem.
\begin{proof}[Proof of Theorem \ref{main}]
Let $\B$ be a bramble in $G$ of maximum order. That is, $\order{\B} = \tw(G)+1$. Let $D'\geq 0$ be a divisor of positive rank and degree $\dgon(G)$. Among the effective divisors equivalent to $D'$, we choose $D$ such that $\supp(D)$ intersects a maximum number of sets in $\B$. If $\supp(D)$ is a hitting set for $\B$, then we are done:
\begin{equation}
\dgon(G)=\deg(D) \geq \supp(D) \geq \order{\B} > \tw(G).
\end{equation}

We may therefore suppose that $B\in\B$ is not intersected by $\supp(D)$ and let $v\in B$. Since $D$ has positive rank  and $D(v)=0$, it follows that $D$ is not $v$-reduced. Hence, by Lemma~\ref{chain}, there exist a chain $\emptyset\subsetneq U_1\subseteq\ldots\subseteq U_k\subseteq V\setminus{\{v\}}$ and a sequence of equivalent effective divisors $D_0:=D,D_1,\ldots,D_k$ such that $D_k$ is $v$-reduced and for every $i=1,\ldots, k$ we have $D_i=D_{i-1}-Q(G)\one_{U_i}$. Since $D$ has positive rank, $\supp(D_k)$ contains $v$ and hence intersects $B$.

Let $i\leq k$ be the smallest index such that there is a $B'\in \B$ that is covered by $\supp(D_0)$ but not by $\supp(D_i)$. Such an index exists, since otherwise $\supp(D_k)$ intersects more members of $\B$ then $\supp(D_0)$, contradicting our choice of $D=D_0$. From  $B'\cap \supp(D_{i-1})\neq \emptyset$ and $B'\cap \supp(D_{i})=\emptyset$ it follows by Lemma \ref{oneside} that $B'\subseteq U_i$. 

Again by our choice of $D$, the set $\supp(D_{i-1})$ does not intersect $B$. Since $\supp(D_k)$ does intersect $B$, there is an index $j\geq i$ such that $B\cap \supp(D_{j-1})=\emptyset$ and $B\cap \supp(D_{j})\neq\emptyset$. Hence, since $D_{j-1}=D_j-Q(G)\one_{V\setminus U_j}$, we have $B\subseteq V\setminus U_j\subseteq V\setminus U_i$ by Lemma \ref{oneside}. 

Since $B\subseteq V\setminus U_i$ and $B'\subseteq U_i$, it follows by Lemma \ref{brambleseparation} that $|E(U_i,V\setminus U_i)|\geq \order{\B}-1$. Since 
\begin{equation}
\deg (D_{i-1})\geq \sum_{u\in U}D_{i-1}(u)\geq |E(U_i,V\setminus U_i)|,
\end{equation}
it follows that $\dgon(G)=\deg(D) = \deg(D_{i-1}) \geq \order{\B} - 1 = \tw(G)$.
\end{proof}

\section{Examples}
We first discuss some classes of graphs for which equality holds in $\tw(G)\leq \dgon(G)$.

\begin{example}
Let $G=(V,E)$ be a \emph{simple} graph with at least one edge. Let $g:=|E|-|V|+1$ be its circuit rank. If $g=0$, then $G$ is a tree and $\tw(G)=\dgon(G)=1$. If $g\in\{1,2\}$, we have $\tw(G)=\dgon(G)=2$. 
\end{example}

\begin{example}[Complete $k$-partite graph]
Let $G=(V,E)$ be a complete $k$-partite graph, $k\geq 2$, with partition $V=V_1\cup\cdots\cup V_k$, where $n_i:=|V_i|\geq 1$. We may assume that $n_1\leq n_2\leq \cdots\leq n_k$. 

For $i=1,\ldots, k$ let $s_i\in V_i$ and consider the bramble $\B:=\{\{s_1\},\ldots,\{s_k\}\}\cup\{\{u,v\}\mid uv\in E\}$. A set $S\subseteq V$ is a hitting set for $\B$ if and only if $s_1,\ldots,s_k\in S$ and there is at most one index $i$ such that $V_i\not\subseteq S$. Hence a hitting set of minimal cardinality is given by $S:=V_1\cup\cdots\cup V_{k-1}\cup\{s_k\}$. Hence $\tw(G)\geq \order{\B}-1=n_1+\cdots+n_{k-1}$.

Let $D:=\one_{V_1\cup\cdots\cup v_{k-1}}$. For every $v\in V_k$, the divisor $D+Q(G)\one_{\{v\}}$ is effective. Hence $D$ has rank at least one and therefore $\dgon(G)\leq n_1+\cdots+n_{k-1}$.

We conclude that $\tw(G)=\dgon(G)=n_1+\cdots+n_{k-1}$. In particular we have $\dgon(K_n)=n-1$ for the complete graph on $n$ vertices, and $\dgon(K_{m,n})=m$ for the complete bipartite graph with colour classes of sizes $m\leq n$. For the octahedron $K_{2,2,2}$ we find $\dgon(K_{2,2,2})=4$.
\end{example}

\begin{example}[Rectangular grid]
Let $m\leq n$ be integers and let $G=(V,E)$ be the $(m+1)\times (n+1)$ rectangular grid. That is, $V:=[m+1]\times [n+1]$ and two vertices $(a,b)$ and $(a',b')$ form an edge if $|a-a'|+|b-b'|=1$. 

Let $A:=[m+1]\times \{n+1\}$ and $B:=\{m+1\}\times [n]$. For $i\in [m]$ and $j\in [n]$ consider the `cross' 
$$C_{ij}:=\{(a,b)\in [m]\times [n]\mid \text{$a=i$ or $b=j$}\}.$$
It is easy to see that $\B:=\{A,B\}\cup\{C_{ij}\}_{i\in [m], j\in [n]}$ is a bramble. Any hitting set for the $C_{ij}$ contains at least $m$ elements from $[m]\times [n]$ (one from each row). Hence, since $A$, $B$, and $[m]\times [n]$ are disjoint, the order of $\B$ is at least $m+2$.

On the other hand, take the divisors $D_i:=\one_{[m+1]\times\{i\}}$ for $i=1,\ldots, m+1$. These divisors are equivalent, since $D_{i+1}=D_i-Q(G)\one_{[m+1]\times [i]}$ for $i=1,\ldots, n$. Hence, since every $(a,b)\in V$ is in the support of some $D_b$, the rank of $D_1$ is at least one. Hence, we can conclude that $m+1\leq \tw(G)\leq \dgon(G)\leq m+1$, and hence $\dgon(G)=\tw(G)=m+1$.   
\end{example}

An interesting family for which we do not know the answer is the following. Let $Q_n$ be the $n$-dimensional cube. That is $Q_n$ is the graph with vertex set $\{0,1\}^n$ and two vertices $x,y$ are connected by an edge if $x$ and $y$ differ in exactly one coordinate. It is clear that $\dgon(Q_n)\leq 2^{n-1}$ and we believe that equality holds. On the other hand, $\tw(Q_n)=\Theta(\frac{2^n}{\sqrt{n}})$, see \cite{SunilKavitha2006}.


\section{Metric graphs}
In this section, we show that for any metric graph $\Gamma$ with underlying connected graph $G$, there is a subdivision $H$ of $G$, such that $\dgon(H)\leq \dgon(\Gamma)$. Hence, the treewidth is also a lower bound for metric graphs: 
\begin{equation}\tw(G)\leq \tw(H)\leq \dgon(H)\leq \dgon(\Gamma).\end{equation}

Let $G$ be a connected graph with vertex set $V$ and edge set $E$. Let $l:E\to \R_{>0}$ be a length function on the edges. Associated to the pair $(G,l)$ is the \emph{metric graph} $\Gamma$ which is the compact connected metric space obtained by identifying edge $e=\{u,v\}$ with a segment of length $l(e)$. The free abelian group on the points of $\Gamma$ is denoted $\Div(\Gamma)$ and the elements of $\Div(\Gamma)$ are called \emph{divisors} on $\Gamma$. For $D=c_1v_1+\cdots+c_kv_k\in \Div(\Gamma)$, with $c_1,\ldots, c_k\in \Z$ and $v_1,\ldots, v_k\in \Gamma$, the \emph{degree} of $D$ is defined as $\deg(D):=c_1+\cdots+c_k$. The divisor is \emph{effective} if $c_1,\ldots, c_k\geq 0$. The support of $D$ is denoted $\supp(D)$.

Let $f$ be a piecewise linear continuous function $f$ on $\Gamma$ with integral slopes. For each $v\in \Gamma$, let $c_v$ be the sum of the outgoing slopes of $f$ at $v$. So $c_v\neq 0$ only for breakpoints of $f$. The associated divisor is denoted $\todiv(f):=\sum_{v\in \Gamma} c_vv$ and is called a \emph{principal} divisor. The set of principal divisors is denoted $\Prin(\Gamma)$ and is a subgroup of $\Div(\Gamma)$. Two divisors are \emph{equivalent} if their difference is a principal divisor. 

For a point $s\in \Gamma$, we say that a divisor $D$ \emph{covers} $s$ if there exists an effective divisor equivalent to $D$ with $v$ in its support. The \emph{gonality} $\dgon(\Gamma)$ is defined as the minimum degree of a divisor that covers every point $v\in \Gamma$. It was proven in \cite{Luo2011} that if $D$ covers every $v\in V$, then $D$ covers every $v\in \Gamma$. However, we will not use that result here.

We denote by $\Div_V(\Gamma)$ the subgroup of divisors with support contained in $V$. We identify the elements of $\Div_V(\Gamma)$ with the corresponding elements on $\Z^V$. Hence, the divisors in $\Div_V(\Gamma)$ can also be seen as divisors on $G$. By $\clattice(\Gamma)$ we denote the set of continuous piecewise linear functions $f$ on $\Gamma$ with integral slopes and $\todiv(f)\in \Div_V(\Gamma)$. This last condition simply means that $f$ is linear on each edge of $\Gamma$. Observe that any two divisors $D,D'\in \Div_V(\Gamma)$ are equivalent if and only if $D-D'=\todiv(f)$ for some $f\in \clattice(\Gamma)$. 

We fix an arbitrary orientation on $G$. We define a map $\phi:\clattice(\Gamma)\to \Z^E$ by setting $\phi(f)(e)$ to be the slope of $f$ on edge $e$ (in the forward direction). Let $g:E\to \Z$. It is easy to see that $g$ is in the image of $\phi$ if and only if
\begin{equation}
\sum_{e\in E} g(e)l(e)\chi_C(e)=0\quad \text{for every cycle $C$ in $G$}.
\end{equation} 

Observe that $\todiv(f)=-M \phi(f)$, where $M$ is the signed vertex-edge incidence matrix of $G$. Also observe that for $l=\one$, the function that is identically one, $g$ is in the image of $\phi$ if and only if $g=M^\transp x$ for some $x\in \Z^V$ by Lemma \ref{clattice}. Hence $\todiv(f)=-M\phi(f)=-MM^\transp x=-Q(G)x$ for some $x\in \Z^v$. In other words, two divisors in $\Div_V(\Gamma)$ are equivalent if and only if they are equivalent as divisors on $G$. 

\begin{theorem}
Let $\Gamma$ be the metric graph associated to $(G,l)$. Then there is a subdivision $H$ of $G$ such that $\dgon(\Gamma)\geq \dgon(H)$. 
\end{theorem}
\begin{proof}
Let $D$ be a minimum degree divisor covering $\Gamma$. In particular, $D$ covers every $v\in V$. Hence, for every $v\in V$, there is an effective divisor $D_v$ equivalent to $D$ with $v$ in its support. Let $V':=V\cup \supp(D)\cup\bigcup_{v\in V}\supp(D_v)$. Let $\Gamma'$ be obtained by subdividing $\Gamma$ at the points in $V'\setminus V$. Denote by $G'$ and $l'$ the corresponding underlying graph and length function so that $\Gamma'$ is the metric graph associated with $(G',l')$. The divisor $D$ and the divisors $D_v$ can now be seen as equivalent elements of $\Div_{V'}(\Gamma')$. 

For all $v\in V$, let $f_v\in \clattice(\Gamma')$ be such that $D-\todiv(f_v)=D_v$. It follows that $y=l'$ is a solution to the system 
\begin{equation}\label{cycle}
\sum_{e\in G'}y(e)\phi(f_v)(e)\chi_C(e)=0\quad \text{for every cycle $C$ in $G$ and every $v\in V$}.
\end{equation} 

Since (\ref{cycle}) is a (finite) rational linear system in $y$, and since $l'>0$ is a solution, the system also has a solution $l''\in \Z_{>0}^E$. It follows that the $D_v$ are equivalent divisors on the metric graph associated with $(G', l'')$. Subdividing every edge $e$ of $G'$ into $l''$ parts to obtain a graph $H$, we can view the $D_v$ as equivalent divisors in $\Div_{V'}(\Gamma'')$, where $\Gamma''$ is the metric graph associated to $(H,\one)$ in which all edges have length one. Finally, this implies that the $D_v$ are also equivalent as divisors of $H$. It follows that for any $v\in V$, the divisor $D_v\in \Div(H)$ covers $V$, and hence by Corollary \ref{rankdetermining} has positive rank.
\end{proof}
The following corollary is immediate.
\begin{corollary}
Let $\Gamma$ be a metric graph with underlying connected graph $G$. Then $\tw(G)\leq \dgon(\Gamma)$.
\end{corollary}

\section{Other notions of gonality}
Other notions of gonality of a graph $G$ have been proposed by Caporaso \cite{Caporaso2012} and by Cornelissen, Kato, and Kool in \cite{CornelissenKatoKool2012}. These notions are based on harmonic morphisms from $G$ to a tree. Here we will show that treewidth is also a lower bound for the gonality in these cases. Again, we assume that our graphs are  connected, finite, and loopless (but possibly with multiple edges).

We follow terminology from \cite{BakerNorine2009}.
A \emph{morphism from $G=(V,E)$ to $G'=(V',E')$}, is a map $\phi:V\cup E\to V'\cup E'$ such that 
\begin{itemize}
\item[(i)] $\phi(V)\subseteq V'$,
\item[(ii)] if $e\in E(u,v)$, then either $\phi(e)=\phi(u)=\phi(v)$, or $\phi(e)\in E'(\phi(u),\phi(v))$.
\end{itemize}
If $\phi(E)\subseteq E'$, then $\phi$ is called a \emph{homomorphism}. We call a morphism $\phi$ \emph{harmonic} if 
\begin{itemize}
\item[(iii)] for every $v\in V$ there exists a nonnegative integer $m_\phi(v)$ such that
\begin{equation}\label{harm}
m_\phi(v)=|\phi^{-1}(e')\cap E(v)|\quad\text{for every $e'\in E'(\phi(v))$,}
\end{equation}
\end{itemize}
\noindent and \emph{non-degenerate} if in addition
\begin{itemize}
\item[(iv)] $m_\phi(v)\geq 1\quad$ for every $v\in V$.
\end{itemize}

\noindent If $\phi$ is harmonic, then there is a number $\deg(\phi)$ such that for every edge $e'\in E'$ and every $v'\in V'$
$$
\deg(\phi)=|\phi^{-1}(e')|=\sum_{v\in \phi^{-1}(v')}m_\phi(v).
$$

\begin{lemma}
Let $G=(V,E)$ and $G'=(V',E')$ be graphs and let $\phi:G\to G'$ be a non-degenerate harmonic morphism. Then $\dgon(G)\leq \dgon(G')\deg(\phi)$. In particular, $\dgon(G)\leq \deg(\phi)$ when $G'$ is a tree.
\end{lemma}
\begin{proof}
For any divisor $D\in \Div(G')$, define the divisor $\phi^*(D)\in \Div(G)$ by $\phi^*(D)(v):=m_\phi(v)D(\phi(v))$. Observe that $\deg(\phi^*(D))=\deg(D)\deg(\phi)$, and that its support is $\phi^{-1}(\supp(D))$ by non-degeneracy of $\phi$. When $D$ is effective, then so is $\phi^*(D)$. 

It is easy to see that if $D,D'\in \Div(G')$ are equivalent, then $\phi^*(D)$ and $\phi^*(D')$ are equivalent as well. Indeed, for any $y\in \Z^{V'}$ we have $\phi^*(Q(G')y)=Q(G)x$, where $x(u):=y(\phi(u))$.

Hence, if $D\in\Div(G')$ is an effective divisor of positive rank in $G'$, then $\phi^*(D)$ is an effective divisor of positive rank in $G$ with $\deg(\phi^*(D))=\deg(D)\deg(\phi)$. 
\end{proof} 
 
The notion of harmonic morphism can be extended to \emph{indexed harmonic morphism} by associating to every edge $e\in \phi^{-1}(E')$ a positive integer $r_e$ and counting in (\ref{harm}) every edge $e\in \phi^{-1}(e')$ with multiplicity $r_e$. Hence, an indexed harmonic morphism $G\to G'$ corresponds to a harmonic morphism $H\to G'$, where $H$ is obtained from $G$ by replacing every edge $e$ by $r_e$ parallel edges which are mapped to the same edge as the original edge $e$.

In \cite{Caporaso2012}, Caporaso defined the gonality of a graph $G$ as the minimum degree of a non-degenerate indexed harmonic morphism (with some additional restriction) from $G$ to a tree. Hence it follows that this measure of gonality is lower bounded by $\dgon(H)$ for some $H$ obtained from $G$ by adding parallel edges, and hence by $\tw(H)=\tw(G)$.

In \cite{CornelissenKatoKool2012}, Cornelissen, Kato and Kool define the \emph{stable gonality} $\sgon(G)$ of $G$ to be the minimum degree of an indexed harmonic homomorphism from a \emph{refinement} of $G$ to a tree $T$. Note that a harmonic homomorphism is automatically non-degenerate. A refinement of $G$ is a graph obtained from $G$ by subdividing edges and adding leaves (nodes of degree 1). Therefore $\sgon(G)$ is lower bounded by $\dgon(H)$ for some graph $H$ obtained from $G$ by subdividing edges, adding leaves and adding parallel edges. Hence, $\sgon(G)\geq \dgon(H)\geq \tw(H)\geq \tw(G)$.

For a comparison of the different notions of gonality, we refer the reader to \cite{CornelissenKatoKool2012}.   

\section{Acknowledgements}
We would like to thank Maarten Derickx. He was the first to prove that the gonality of the $n\times m$ grid equals $\min(m,n)$ (unpublished). His method inspired us to conjecture and prove that $\tw(G)\leq \dgon(G)$. The second author would also like to thank Jan Draisma for some stimulating discussions on graph gonality.

\bibliographystyle{plain}
\bibliography{gonality,treewidth,graphs}

\end{document}